\documentclass{amsart}
\usepackage{amsmath}
\usepackage{amssymb}
\usepackage[all]{xy}
\usepackage{graphicx}
\usepackage{mathrsfs}
\usepackage{enumitem}

\numberwithin{equation}{section}

\newtheorem{theorem}{Theorem}[section]
\newtheorem{lemma}[theorem]{Lemma}

\newtheorem{definition}[theorem]{Definition}

\newtheorem{proposition}[theorem]{Proposition}
\newtheorem{corollary}[theorem]{Corollary}

\theoremstyle{remark}
\newtheorem{remark}[theorem]{Remark}

\numberwithin{equation}{section}

\newcommand{\arxiv}[1]{\href{http://arxiv.org/abs/#1}{\tt arXiv:\nolinkurl{#1}}}
\newcommand{\arXiv}[1]{\href{http://arxiv.org/abs/#1}{\tt arXiv:\nolinkurl{#1}}}

\usepackage[latin1]{inputenc}
\usepackage{xspace,amssymb,amsfonts,euscript}
\usepackage{amsthm,amsmath}
\usepackage{palatino}
\usepackage{euscript}
\input xy \xyoption {all}

\usepackage{tikz}

\RequirePackage{color}
\definecolor{myred}{rgb}{0.75,0,0}
\definecolor{mygreen}{rgb}{0,0.5,0}
\definecolor{myblue}{rgb}{0,0,0.65}

\RequirePackage{ifpdf}
\ifpdf
  \IfFileExists{pdfsync.sty}{\RequirePackage{pdfsync}}{}
  \RequirePackage[pdftex,
   colorlinks = true,
   urlcolor = myblue, 
   citecolor = mygreen, 
   linkcolor = myred, 
   pagebackref,
   bookmarksopen=true]{hyperref}
\else
  \RequirePackage[hypertex]{hyperref}
\fi

\newcommand{\nc}{\newcommand}

\nc{\flags}{\mathcal{F}}

\def\ii{{\bf i}}

\def\N{\mathbb{N}}
\def\Q{\mathbb{Q}}
\def\C{\mathbb{C}}
\def\R{\mathbb{R}}
\def\Z{\mathbb{Z}}
\def\F{\mathcal{F}}

\def\L{\mathcal{L}}

\def\O{\mathcal{O}}

\def\jj{{\bf j}}

\def\P{\mathcal{P}} 
\def\vv{\mathbf v}
\def\vw{\mathbf w}

\def\uk{\underline{k}}

\def\a{\alpha}
\def\b{\beta}

\def\la{\lambda}

\def\ga{\gamma}

\def\w{\omega}
\def\th{\theta}

\def\inv{^{-1}}

\DeclareMathOperator{\KP}{KP}

\def\hd{\operatorname{hd}}
\def\f{\mathbf{f}}

\def\G{\mathcal{G}}

\def\D{\Delta}

\nc{\ic}{\mathbf{IC}}

\def\s{\sigma}

\def\op{{\mathrm{op}}}\def\pt{{\mathrm{pt}}}

\DeclareMathOperator{\Hom}{Hom}
\DeclareMathOperator{\Ext}{Ext}
\def\homb{\Hom^\bullet}

\def\Ind{\operatorname{Ind}}

\def\End{\operatorname{End}}

\def\Res{\operatorname{Res}}

\def\supp{\operatorname{supp}}

\def\mods{\mbox{-mod}}

\def\prmod{\mbox{-pmod}}

\newcommand{\map}[2]{\,{:}\,#1\!\longrightarrow\!#2}
\newcommand{\spann}{\operatorname{span}}

\makeatletter
\newcommand{\mylabel}[2]{#2\def\@currentlabel{#2}\label{#1}}
\makeatother

\def\kpf{\operatorname{kpf}} 

\def\inv{^{-1}}

\def\T{\mathcal{T}}
\def\g{\mathfrak{g}}

\numberwithin{equation}{section}

\title{Representation Theory of Geometric Extension Algebras}

\author[Peter J McNamara]{Peter J McNamara}
\address{School of Mathematics and Statistics,
The University of Melbourne, VIC, Australia.}
\email{maths@petermc.net}

\date{\today}

\begin{document}

\begin{abstract}
We study the question of when geometric extension algebras are polynomial quasihereditary. Our main theorem is that under certain assumptions, a geometric extension algebra is polynomial quasihereditary if and only if it arises from what we call an even resolution. We give an application to the construction of reflection functors for quiver Hecke algebras.
\end{abstract}

\maketitle

\tableofcontents

\section{Introduction}

This paper is about the study of algebras which appear geometrically in the form
\[
A(X,f):= \Hom_{D^b_G(X)}^\bullet (f_*\uk_Y,f_*\uk_Y)
\]
where $f\map{Y}{X}$ is a proper $G$-equivariant morphism of complex algebraic varieties with $Y$ smooth, and $k$ is the coefficient ring. Such algebras have a history of appearing in Lie-theoretic representation theory, dating back to \cite{chrissginzburg,lusztig}. 
For our applications, the most important examples are the quiver Hecke algebras (also known as KLR algebras) of finite ADE type.

This paper is heavily influenced by the ideas of Kato in his papers \cite{katoext,katopbw}, where such algebras were studied in the case where $k=\overline{\Q}_l$.
Kato proved that certain such extension algebras are polynomial quasihereditary, and studied some consequences of this fact.

We work over a general ground field in a deliberate attempt for this work to be relevant for modular representation theory. This requires us to replace the role of purity arguments from Kato's work with arguments based on the notion of evenness.

Our main theorems are Theorems \ref{thm:fwd} and \ref{reverse}. Theorem \ref{thm:fwd} states that under conditions $(\spadesuit)_1$, $(\spadesuit)_2$ and $(\heartsuit)$ (to be introduced in \S\ref{sec3}), a geometric extension algebra for an even morphism $f$ is polynomial quasihereditary. Moreover, there is a converse, which is Theorem \ref{reverse}: If $A(X,f)$ is concentrated in even degrees and is polynomial quasihereditary, then $f$ is even.

After proving our main theorem, we turn our attention to an application to the theory of quiver Hecke algebras. The fact that finite type quiver Hecke algebras are polynomial quasihereditary is essentially due to \cite{bkm}, we give a complete proof in Theorem \ref{thm:klrhw}. As a corollary, we deduce the fact that the corresponding morphism $f$ is even. A geometric proof of this result has subsequently been discovered by Maksimau \cite{rm19}. This allows us access to the machinery of parity sheaves \cite{parity} in studying the representation theory of these quiver Hecke algebras. We utilise this machinery to generalise Kato's construction of reflection functors categorifying the braid group action on $U_q(\g)$.

%

We thank K. Coulembier, P. Shan, T. Braden, C. Mautner and S. Makisumi for useful conversations, and some anonymous referees for their comments.

\section{Polynomial quasihereditary algebras}

In this section, we collect various facts about polynomial quasihereditary algebras which we will need. The main reference is \cite{kleshchev}.

A $\Z$-graded algebra $A=\oplus_{n\in\Z} A_n$ is \emph{Laurentian} if $A_n$ is finite dimensional for all $n$, and there exists $N\in\Z$ such that $A_n=0$ for $n<N$.
When considering modules over a graded algebra, the assumption which we will always make in this paper is that all modules we consider are always graded modules. We use multiplication by $q$ to indicate the grading shift.	

A \emph{polynomial ring} is a commutative graded ring $k[x_1,x_2,\ldots,x_n]$, where each variable $x_i$ is homogeneous of positive degree. A standard example is $H^*_{GL_n(\C)}(\pt)$.

Let $A$ be a Noetherian Laurentian graded unital algebra. Then $A$ has a finite number of simple modules up to isomorphism and grading shift, all of which are finite  dimensional.
Let $\Pi$ be an indexing set for this set of simples. For each $\pi\in\Pi$, let $L(\pi)$ be a choice of the corresponding simple and let $P(\pi)$ be its projective cover.

\begin{definition}
 A two sided ideal $J\subset A$ is called polynomial heredity if it satisfies the following conditions:
\begin{enumerate}
\item[{\tt (SI1)}] $\Hom_A(J,A/J)=0$;
\item[{\tt (SI2)}] As a left module,  
$J\cong P(\pi)^{\oplus m}$ for some $\pi\in\Pi$ and $m\in \N[q,q\inv]$, and setting $B(\pi):=\End_A(P(\pi))^{\op}$, we have $B(\pi)$ is a graded polynomial ring.
\item[{\tt (PSI)}] As a right $B(\pi)$-module, $P(\pi)$ is finitely generated and free.
\end{enumerate}
\end{definition}

Since $B(\pi)$ is assumed to be a polynomial ring, condition {\tt (PSI)} is equivalent to $P(\pi)$ being finitely generated and flat as a right $B(\pi)$-module. It is this latter condition which appears in \cite[Definition 6.1]{kleshchev}.

\begin{definition}
 The algebra $A$ is called polynomial quasihereditary if there exists a finite chain of two-sided ideals
 $
 A=J_0\supsetneq J_1\supsetneq\cdots\supsetneq J_n=(0)
 $
 such that $J_i/J_{i+1}$ is a polynomial heredity ideal in $A/J_{i+1}$ for all $0\leq i<n$.
\end{definition}

Let $\prec$ be a partial order on $\Pi$. For each $\pi\in\Pi$, let $\Sigma_\pi$ be the full subcategory of $A$-modules whose Jordan-Holder constituents are all of the form $q^nL(\sigma)$ where $\sigma\succeq \pi$.
For each $\pi\in\Pi$, let $\Delta(\pi)$ be the projective cover of $L(\pi)$ in $\Sigma_\pi$.

\begin{definition}
 The category $A\mods$ is polynomial highest weight if the following conditions are satisfied for each $\pi\in \Pi$:
 \begin{enumerate}
  \item[\mylabel{sc1}{\tt (SC1)}] $P(\pi)$ has a filtration $0=P_0\subset P_1\subset\cdots\subset P_n=P(\pi)$ with $P_n/P_{n-1}\cong \D(\pi)$ and for $1\leq i< n$, $P_i/P_{i-1}\cong q^N\Delta(\sigma)$ for some $N\in\Z$ and $\sigma\prec\pi$.
 \item[\mylabel{sc2}{\tt (SC2)}] $B(\pi):=\End_A(\D(\pi))^{\rm{op}}$ is a polynomial algebra.
 \item[\mylabel{hwc}{\tt (HWC)}] $\D(\pi)$ is a finitely generated free right $B(\pi)$-module.
 \end{enumerate}
\end{definition}

In \cite{kleshchev}, one will see in place of \ref{hwc} the condition that $\Hom(P(\sigma),\D(\pi))$ is a free right $B(\pi)$-module of finite rank. This follows from \ref{hwc} since $\Hom(P(\sigma),\D(\pi))$ is a direct summand of $\Hom(A,\D(\pi))\cong\D(\pi)$. Therefore under \ref{hwc}\!, $\Hom(P(\sigma),\D(\pi))$ is a summand of a finite free $B(\pi)$-module, hence is finite and free since $B(\pi)$ is a polynomial algebra.

Kleshchev proves the following result, which in the finite dimensional setting is the classical result of Cline, Parshall and Scott \cite[Theorem 3.6]{cps}.

\begin{theorem}\cite[Theorem 6.7]{kleshchev}\label{affinecps}
 The algebra $A$ is polynomial quasihereditary if and only if the category $A\mods$ is polynomial highest weight.
\end{theorem}

\section{Geometric setup}\label{sec3}

Let $X$ be a complex algebraic variety with an action of an algebraic group $G$. For each $x\in X$, let $G_x$ be the stabiliser of $x$. Let $k$ be a field. We assume the following conditions:

\begin{itemize}
 \item[$(\spadesuit)_1$] $G$ acts on $X$ with finitely many orbits.
 \item[$(\spadesuit)_2$] For all $x\in X$, the stabiliser $G_x$ is connected and $H^i_{G_x}({\rm{pt}};k)$ vanishes for odd $i$.
\end{itemize}

 Condition $(\spadesuit)_2$ always holds when $G_x$ is a connected affine algebraic group whose quotient by its unipotent radical is a product of groups $GL_n(\mathbb{C})$.
In general, if $G_x$ is a connected affine algebraic group, $(\spadesuit)_2$ holds whenever $\operatorname{char}(k)$ is outside a small set of known primes, called torsion primes for $G_x$. Condition $(\spadesuit)_2$ implies that $H^\bullet_{G_x}(\rm{pt};k)$ is a polynomial algebra. See \cite[Theorem 2.44]{parity} and the references quoted there.

We spend most of our time working inside the $G$-equivariant derived category $D_G(X;k)$ of constructible $k$-sheaves on $X$, whose definition and basic properties can be found in \cite{bernsteinlunts}. For any objects $\F$ and $\G$, we write
\[
 \Hom^\bullet(\F,\G):=\bigoplus_{d\in\Z} \Hom_{D_G(X;k)}(\F,\G[d])
\]
which is considered as a graded vector space in the obvious way.

Let $f\map{Y}{X}$ be a proper $G$-equivariant morphism, where $Y$ is a smooth complex algebraic variety. Let $\L=f_*\uk_Y\in D^b_G(X)$. We will study the representation theory of the graded algebra
\[
  A(X,f):= \Hom^\bullet(\L,\L).
\]
We only consider $A(X,f)$ as an associative algebra in this paper, we do not care whether it arises from a formal $A_\infty$-algebra. If $k$ has characteristic zero, these algebras will often be formal at least in cases of representation-theoretic interest (for an example, see \cite[Lemma 4.7]{webster}). 

\begin{lemma}
 The algebra $A(X,f)$ is Noetherian and Laurentian.
\end{lemma}

\begin{proof}
 Let $Z=Y\times_X Y$. Then $A(X,f)$ is isomorphic to the $G$-equivariant Borel-Moore homology of $Z$. Since $Z$ is an algebraic variety, $A(X,f)$ is therefore finitely generated over $H^\bullet_G(\pt)$, hence is Laurentian and Noetherian.
\end{proof}

We make one more assumption.

\begin{itemize}
 \item[$(\heartsuit)$] Taking the closure of the support induces a bijection between indecomposable direct summands of $f_*\uk_Y$ up to shift and the set of $G$-orbits on $X$. The restriction of each indecomposable summand to the corresponding orbit is the constant sheaf on the orbit.
\end{itemize}

We will check this strange looking condition for the quiver Hecke algebras in Theorem \ref{thm:checkheartsuit}.

Let $\Pi$ be a set indexing the orbits. We write $\O_\pi$ for the orbit corresponding to $\pi\in \Pi$ and $\P_\pi$ for the corresponding indecomposable summand of $\L$. We may normalise the homological shift so that the restriction of $\P_\pi$ to $\O_\pi$ is a sheaf.
Conversely, given an orbit $Z$, we write $\pi_Z$ for the corresponding element of $\Pi$.
Furthermore, define
\[
 P(\pi)=\Hom^\bullet (\L,\P_\pi)\quad \mbox{and} \quad L(\pi)=\hd P(\pi).
\]
\begin{theorem}
 The modules $L(\pi)$ comprise a complete set of irreducible modules for $A(X,f)$, up to grading shift. The projective cover of $L(\pi)$ is $P(\pi)$.
\end{theorem}

\begin{proof}
 The category $D^b_G(X;k)$ is Karoubian and Krull-Schmidt. 
 Therefore 
 the additive envelope of $\L$ is equivalent to the category of projective $A(X,f)$-modules. The additive envelope of $\L$ is the smallest full subcategory containing $\L$ that is closed under shifts, direct sums and direct summands.
 This equivalence is given by $\Hom^\bullet(\L,-)$, so we see that the $P(\pi)$ are the indecomposable projectives, proving the theorem.
\end{proof}

%

Let $\F\in D^b_G(X;k)$. $\F$ is said to be $*$-even if $H^j(i_x^*\F)=0$ for all odd $j$ and all inclusions $i_x\map{\{x\}}{X}$ of a point into $X$. $\F$ is said to be $!$-even if $H^j(i_x^!\F)=0$ for all odd $j$ and all inclusions $i_x\map{\{x\}}{X}$ of a point into $X$. $\F$ is said to be even if it is both $*$-even and !-even.
For $?\in \{*,!,\emptyset\}$, $\F$ is $?$-parity if $\F\cong \F'\oplus \F''[1]$ with $\F'$ and $\F''$ $?$-even (this corresponds to the zero pariversity function in the terminology of \cite{parity}).

For each $x\in X$, let $Y_x$ be the fibre $f\inv(x)$. We require the following notions of evenness, which are primarily useful for proper morphisms.

\begin{definition}
 The morphism $f$ is even if $H_i(Y_x;k)=0$ for all $x\in X$ and odd integers $i$.
\end{definition}

A $\Z$-graded algebra $A=\oplus_{d\in\Z}A_d$ is said to be even if $A_d=0$ for all odd $d$.

\begin{lemma}\label{evenimplieseven}
 If $f$ is even, then $A(X,f)$ is even.
\end{lemma}

\begin{proof}
 By \cite[Proposition 2.34]{parity}, $\L=f_*\uk_Y$ is even. By \cite[Proposition 2.6]{parity}, this implies that $A(X,f)$ is even.
\end{proof}

\begin{lemma}\label{evensplit}
 Let $\F$ be an even object in $D^b_H(\pt;k)$, where $H$ is a connected algebraic group with $H^*_H(\pt;k)$ concentrated in even degrees. Then $$\F\cong \bigoplus_{d\in 2\Z} \underline{k}[d]^{\oplus m_d}$$ for some integers $m_d$.
\end{lemma}

\begin{proof}
 We proceed by induction on the width of the support of $\F$. This support is the set of integers $i$ such that $H^i(\F)\neq 0$, where $H^i$ is the cohomological functor associated with the usual t-structure. If this width is one, then $\F$ lies (up to a shift) in the abelian category of $H$-equivariant sheaves on a point. Since $H$ is connected, every such sheaf is a direct sum of constant sheaves. This completes the proof in this case.
 
 Now assume that the width of $\F$ is greater than one. Consider a triangle of the form
 \[
  \tau_{\leq i}\F\to \F\to \tau_{>i} \F\xrightarrow{+1}
 \]
Choose $i$ such that $\tau_{\leq i}\F$ and $\tau_{>i}\F$ have smaller support than $\F$, so by the inductive hypothesis this triangle
is of the form
\[
 \bigoplus_{{d\in 2\Z, d\leq i}} \uk[d]^{\oplus m_d} \to \F \to  \bigoplus_{{d\in 2\Z, d> i}} \uk[d]^{\oplus m_d} \xrightarrow{+1}.
\]
Since $H_H^*(\pt)$ is concentrated in even degrees, $\Hom(\uk,\uk[e])=0$ for odd $e$. Therefore the morphism $\tau_{>i}\F\to \tau_{\leq i}\F[1]$ is zero. Thus $\F\cong \tau_{\leq i}\F\oplus \tau_{>i}\F$, hence is of the desired form.
\end{proof}

\section{Evenness implies affine highest weight}

The aim of this section is to prove Theorem \ref{thm:fwd}. Throughout, the partial order on $\Pi$ is the closure relation.
We assume throughout this section that $\L$ is even.

Let $\pi\in\Pi$. Write $o_\pi\map{\O_\pi}{X}$ for the inclusion of the corresponding orbit. Define
\[
 \D(\pi)=\Hom^\bullet (\L,(o_\pi)_*\uk_{\O_\pi}).
\]
This is an $A(X,f)$-module. We will show in Proposition \ref{algstd} that this choice of notation is justified, in the sense that this agrees with the algebraically defined module with the same name. Until we make this identification, we reserve the notation $\D(\pi)$ for this geometrically defined module.

%

\begin{lemma}\label{parityses}
 Let $\L$ be parity and $\P$ be !-even on $X$. Let $i\map{Z}{X}$ be the inclusion of a closed subset and $j$ the inclusion of the open complement. Then $i_!i^!\P$ and $j_*j^*\P$ are also !-even, and there is a short exact sequence of $\Hom^\bullet(\L,\L)$-modules
 \[
  0\to \Hom^\bullet(\L,i_!i^!\P)\to \homb(\L,\P)\to \homb(\L,j_*j^*\P)\to 0.
 \]
\end{lemma}

\begin{proof}
 It follows from \cite[Proposition 2.6]{parity} that the long exact sequence obtained by applying $\Hom(\L,-)$ to the triangle $i_!i^!\P\to \P\to j_*j^*\P\xrightarrow{+1}$ breaks up into a series of short exact sequences for parity vanishing reasons.
\end{proof}

\begin{corollary}\cite[Corollary 2.9]{parity}\label{cor:fsurj}
 Let $j\map{U}{X}$ be the inclusion of an open $G$-stable subset and $\L$ be a parity sheaf on $X$. Then the canonical algebra homomorphism
 \[
  \Hom^\bullet(\L,\L)\rightarrow \Hom^\bullet(j^*\L,j^*\L)
 \]
is surjective.
\end{corollary}

\begin{proof}
 Take $\P=\L$ in Lemma \ref{parityses} and use the $(j^*,j_*)$ adjunction.
\end{proof}

A module $M$ is said to have a \emph{standard flag} if it has a finite flag of submodules $0=M_0\subset M_1\subset \cdots \subset M_n=M$ such that each quotient $M_{i}/M_{i-1}$ is isomorphic to $q^N\D(\pi)$ for some $N\in \Z$ and $\pi\in \Pi$.

\begin{lemma}\label{standardflag}
 Let $\F$ be 
 !-even. Then $\Hom^\bullet (\L,\F)$ has a standard flag.
\end{lemma}

\begin{proof}
 Induct on the number of orbits in $\supp(\F)$. Let $Z\subset \supp(\F)$ be closed such that $V:=\supp(F)\setminus Z$ is a single orbit. Let $U=X\setminus Z$ and write $j\map{U}{X}$ and $i\map{Z}{X}$ for the inclusions. 
 Then $i_!i^!\F$ and $j_*j^*\F$ are both 
 !-even. So there is a short exact sequence of $A(X,f)$-modules
 \[
  0 \to \Hom^\bullet (\L,i_!i^!\F)\to \Hom^\bullet (\L,\F)\to \Hom^\bullet(\L,j_*j^*\F)\to 0.
 \]
Let $o\map{V}{X}$ be the inclusion. We have $j_*j^*\F\cong o_*o^*\F$. Since $o^*\F$ is !-even and $V$ is a single orbit, it is even and Lemma \ref{evensplit} implies that $o^*\F$ is a direct sum of even shifts of $\uk_V$. Therefore $\Hom^\bullet(\L,j_*j^*\F)$ is a direct sum of even shifts of the standard module $\D(\pi_V)$. By inductive assumption, $\Hom^\bullet(\L,i_!i^!\F)$ has a standard filtration. Therefore $\Hom^\bullet(\L,\F)$ has a standard filtration. 
\end{proof}

\begin{corollary}\label{proofofsc1}
The condition \ref{sc1} holds.
\end{corollary}

\begin{proof}
 The projective $P(\pi)$ is $\Hom^\bullet (\L,\P_\pi)$ with $\P_\pi$ parity. Therefore $P(\pi)$ has a standard filtration. Since $\supp(\P_\pi)=\overline{\O_\pi}$ and $\P_\pi|_{\O_\pi}\cong \uk$, an examination of the proof of Lemma \ref{standardflag} shows that all sections of the filtration bar one are of the form $\D(\sigma)$ with $\sigma \prec\pi$, and the remaining section is $\D(\pi)$ and appears as a quotient.
\end{proof}

%
%

\begin{proposition}\label{algstd}
 The module $\D(\pi)$ is an algebraic standard module, in the sense that it is the projective cover of $L(\pi)$ in $\Sigma_\pi$.
\end{proposition}

\begin{proof}
 To achieve this identification, we have to show that
\begin{enumerate}
 \item \label{cond:1} Any simple subquotient of $\D(\pi)$ is of the form $L(\sigma)$ with $\sigma\succeq \pi$.
 \item \label{cond:2} If $\sigma\succeq\pi$, then $\dim\Hom(\D(\pi),L(\sigma))=\delta_{\pi\sigma}$ and $\Ext^1(\D(\pi),L(\sigma))=0$.
\end{enumerate}
 
 For the first fact, suppose that $L(\sigma)$ is a subquotient of $\D(\pi)$. Then there is a nontrivial homomorphism from $P(\sigma)$ to $\Delta(\pi)$, hence the projection from $\L$ to the summand $\P_\sigma$ acts nontrivially on $\D(\pi)$. Therefore
 \begin{equation}\label{asdf}
  \Hom^\bullet(\P_\sigma,(o_\pi)_*\uk)\neq 0.
 \end{equation}
Unless $\sigma\succeq \pi$, there is an open $G$-invariant subset $U$ of $X$ containing $\O_\pi$ such that $\supp(\P_\sigma)\cap U=\emptyset$. Write $j\map{U}{X}$ for the open embedding. Then $j^*\P_\sigma=0$ so
\[
 \Hom^\bullet(\P_\sigma,j_*\F)=\Hom^\bullet(j^*\P_\sigma,\F)=0
\]
for any $\F$. Since $(o_\pi)_*\uk$ is of the form $j_*\F$, this contradicts (\ref{asdf}) above, proving the first part.

Now consider the short exact sequence from Corollary \ref{proofofsc1}
\begin{equation}\label{kpd}
 0\to K(\pi)\to \P(\pi)\to \D(\pi)\to 0
\end{equation}
where $K(\pi)$ has a filtration with subquotients of the form $\D(\tau)$, for $\tau\prec\pi$.

Each $\D(\tau)$ is a quotient of $P(\tau)$, so has an irreducible head, namely $L(\tau)$. Therefore $\Hom(K(\pi),L(\sigma))=0$ unless $\sigma\prec\pi$. Applying $\Hom(-,L(\sigma))$ to the short exact sequence (\ref{kpd}) then yields $\Ext^1(\D(\pi),L(\sigma))=0$ unless $\sigma\prec\pi$, as required.
\end{proof}

\begin{proposition}
 The module $\D(\pi)$ satisfies \ref{sc2} and \ref{hwc}.
\end{proposition}

\begin{proof}
Let $j\map{U}{X}$ be the inclusion of an open $G$-stable subset such that $\O_\pi$ is closed in $U$. 
 The module $\D(\pi)$ is of the form $\homb(\L,j_*\F)=\Hom^\bullet(j^*\L,\F)$ and the action of $\Hom^\bullet(\L,\L)$ factors through the quotient $\Hom^\bullet(j^*\L,j^*\L)$.
 Therefore in checking conditions \ref{sc2} and \ref{hwc}, it suffices to assume without loss of generality that $\O_\pi$ is closed in $X$.

Now write $Z$ for $\O_\pi$, $i\map{Z}{X}$ for the inclusion and suppose $Z$ is closed in $X$. Then by condition ($\heartsuit$), $\P_\pi\cong i_*\uk_Z$ up to a shift.
 Recalling that $P(\pi)=\Hom^\bullet (\L,\P_\pi)$, we get $$\End(P(\pi))\cong \Hom^\bullet(i_*\underline{k}_Z,i_*\underline{k}_Z)\cong H^*_{G}(Z).$$ 
 By $(\spadesuit)_2$, this is a polynomial algebra, proving \ref{sc2}.

 Let $Y_Z=f\inv(Z)$.
Since $f$ is $G$-equivariant, the map $Y_Z\to Z$ is a fibration with fibre $f\inv(z)$, where $z$ is some point in $Z$. We consider the Leray spectral sequence for $G$-equivariant sheaf cohomology applied to the dualising sheaf of $Y_Z$. There is no monodromy to worry about because of $(\spadesuit)_2$, so this spectral sequence takes the form
\[
 H^*_G(Z;H^{}_* (f\inv(z)))\implies H^{G}_*(Y_Z).
\]
Here we write $H_*$ for Borel-Moore homology and $H^G_*$ for $G$-equivariant Borel-Moore homology.
Since $f$ is even, this spectral sequence degenerates at the $E_2$ page. The $E_\infty$ page is $P(\pi)\cong H^{G}_*(Y_Z)$. Therefore $P(\pi)$ is free over $\End(P(\pi))\cong H^*_G(Z)$ since the $E_2$ page is obviously free, showing condition {\tt (HWC)}.
\end{proof}

We have thus proved

\begin{theorem}\label{thm:fwd}
 Assume $(\spadesuit)_1$, $(\spadesuit)_2$ and $(\heartsuit)$. Suppose that $f$ is even. Then $A(X,f)\mods$ is polynomial highest weight for the partial order given by orbit closures.
\end{theorem}

As a consequence, these algebras satisfy the properties developed in \cite{kleshchev} and \cite{fujita}. For example they have finite global dimension and a family of proper standard modules. If they are finite over their centre (which we conjecture is always the case and is known for quiver Hecke algebras \cite[Corollary 2.10]{khovanovlauda}), then they also admit a theory of tilting modules.

\section{The reverse implication}

The aim of this section is to prove our other main theorem, which is a partial converse to Theorem \ref{thm:fwd}. First we record a lemma about spectral sequences.

\begin{lemma}\label{ssdegen}
Consider a cohomological spectral sequence concentrated in the first quadrant that converges to $\D$. Let $B=E_2^{\bullet,0}$. Suppose that $B$ is a polynomial ring and that $B$ acts freely on $\D$. Then the spectral sequence degenerates at the $E_2$-page.
\end{lemma}

\begin{proof}
Write $F_0\subset F_1\subset \cdots$ for the filtration of $\D$ such that $F_i/F_{i-1}\cong E_\infty^{\bullet,i}$. Each $F_i$ is a $B$-module.

We will prove by strong induction on $i$ that $F_i$ is a free $B$-module, generated in degrees at most $i$ and that $E_2^{\bullet,i}=E_\infty^{\bullet,i}$.

Assume then that $b$ is such that the inductive hypothesis holds for $i<b$. Since $F_b/F_{b-1}$ is generated in degrees at least $b$, we have $F_b\cong F_{b-1}\oplus F_b/F_{b-1}$. As $F_b$ is a submodule of $\D$, this implies that $F_b/F_{b-1}$ is torsion-free.

The $B$-module $E_2^{\bullet,b}$ is free over $B$ and generated by $E_2^{0,b}$. By the inductive hypothesis on the spectral sequence degenerating in rows below $b$, we have $E_2^{0,b}=E_\infty^{0,b}$. The presence of any non-trivial differential involving the $b$-th row of the spectral sequence will therefore cause $E_\infty^{\bullet,b}$ to not be torsion-free over $B$, i.e. $F_b/F_{b-1}$ to not be torsion-free over $B$. This is a contradiction. Therefore every differential involving the $b$-th row is trivial.

We have thus shown that $E_2^{\bullet,b}=E_\infty^{\bullet,b}$ and therefore both are free over $B$, generated in degree $b$. This completes the proof of the induction hypothesis for $i=b$, so we have finished the proof of the lemma.
\end{proof}

Here is our other main theorem:

\begin{theorem}\label{reverse}
 Assume $(\spadesuit)_1$, $(\spadesuit)_2$, $(\heartsuit)$ and that $A(X,f)$ is even. Suppoose that $A(X,f)$ is polynomial quasihereditary with respect to the partial order induced by the closure order between orbits.
 Then $f$ is even.
\end{theorem}

\begin{remark}
 We only need to know that $A(X,f)$ is polynomial quasihereditary for some refinement of $\prec$ to a total order.
\end{remark}

\begin{proof}
 We induct on the number of orbits of $G$ on $X$. Let $i\map{Z}{X}$ be the inclusion of a closed orbit. Let $z$ be a point in $Z$. Let $Y_Z=f\inv(Z)$.
 
 By condition $(\heartsuit)$, $i_*\uk_Z$ occurs up to grading shift as a direct summand of $f_*\uk_Y$. Therefore the projective module corresponding to $Z$ is $P_Z:=\homb(\L,i_*\uk_Z)$. Using the properness of $f$ to write $f_*=f_!$ and the usual adjunction and base-change properties, this can also be written as a grading shift of $H_*^{G}(Y_Z).$

Since $Z$ is minimal in the order $\prec$, this module is also the standard module $\D_Z$ associated to the orbit $Z$.
 
The endomorphism algebra of $\D_Z=P_Z$ is $$B_Z:=(\End \D_Z)^\op\cong \homb (i_*\uk_Z,i_*\uk_Z)\cong H^*_G(Z).$$

Again consider the spectral sequence
\[
 H^*_G(Z;H^{}_* (f\inv(z)))\implies H^{G}_*(Y_Z)=\D_Z.
\]
By Lemma \ref{ssdegen}, this spectral sequence degenerates at the $E_2$ page. 

Thus if $f\inv(z)$ fails to be even, then $P_Z=\D_Z$ also fails to be even. Since $P_Z$ is a direct summand up to shift of $A(X,f)$ which is assumed to be even, this proves that $f\inv(z)$ is even.

Now consider $i^!\L$. By base change and smoothness of $Y$, we have $i^!\L=f_*\w_{Y_Z}[-2\dim Y]$. Since $f\map{Y_Z}{Z}$ is a fibration whose fibre has no odd Borel-Moore homology, the object $i^!\L$ is an even object of $D^b_G(Z;k)\cong D_{G_z}^b(\pt;k)$.
By Lemma \ref{evensplit}, it is thus a direct sum of even shifts of $\uk_Z$.

Therefore $\homb(\L,i_!i^!\L)$ is a direct sum of even shifts of $P_Z$. We denote this module $P$.

Let $U=X\setminus Z$ and $j\map{U}{X}$ be the open embedding. Let $$Q_0=\bigoplus_{d\in 2\Z}\Hom(j^*\L,j^*\L[d])\quad \mbox{and}\quad Q_1=\bigoplus_{d\in 2\Z+1}\Hom(j^*\L,j^*\L[d]).$$ We apply $\Hom(\L,-)$ to the triangle
\[
  i_!i^!\L\to \L \to j_*j^*\L \xrightarrow{+1}
\]
and obtain the following exact sequence of $A(X,f)$-modules:
\begin{equation}\label{qses}
 0\to Q_{1} \to P \to A(X,f)\to Q_{0}\to 0
\end{equation}
where $Q_0=\Hom^\bullet(j^*\L,j^*\L)$.

Let $J$ be the heredity ideal in $A(X,f)$ corresponding to the closed orbit $Z$. We first aim to show that $Q_0=A(X,f)/J$. To do this, we need to show
\begin{enumerate}
 \item $\Hom(Q_0,L(\pi_Z))= 0$
 \item $\Hom(Q_0,L(\pi))\neq 0$ for all $\pi\neq\pi_Z$.
 \item $\Ext^1(Q_0,L(\pi))=0$ for all $\pi\neq \pi_Z$.
\end{enumerate}

For (1), suppose there is a nonzero homomorphism from $Q_0$ to $L(\pi_Z)$. It is surjective as $L(\pi_Z)$ is simple, so we obtain a nonzero homomorphism from $P_Z$ to $Q_0$. Therefore there is a nonzero composition $\L\to i_*\uk_Z\to \L\to j_*j^*\L[d]$ for some $d$. But
\[
 \Hom(i_*\uk_Z,j_*j^*\L[d])\cong \Hom(j^*i_*\uk_Z,j^*\L[d])=0
\]
since $j^*i_*=0$, a contradiction.

For (2), if $\pi\neq \pi_Z$, then $j^*\P_\pi\neq 0$. Then $\Hom^\bullet (j^*\L,j^*\P_\pi)$ is a summand of $Q_0$ which is nonzero since it contains the identity morphism on $j^*\P_\pi$. Using (\ref{qses}), we see there is a surjection from $P(\pi)$ to $\Hom^\bullet (j^*\L,j^*\P_\pi)$. Therefore $\Hom^\bullet (j^*\L,j^*\P_\pi)$ has $L(\pi)$ as a quotient, hence the same is true of $Q_0$.

For (3), first note that since $A(X,f)$ is projective, the sequence $(\ref{qses})$ shows that
$\Ext^1(Q_{0},L(\pi))$ is a subquotient of $\Hom(P,L(\pi))$.
Since $P$ is a direct sum of shifts of $P_Z$, we have $\Hom(P,L(\pi))=0$ for all $\pi\neq \pi_Z$, as required.

Therefore $Q_{0}$ is the quotient by the heredity ideal $J$ associated to $Z$. So there is a short exact sequence
\[
 0\to Q_1\to P\to J\to 0.
\]

We have already shown that $P$ is a direct sum of copies of $P_Z$. Since $J$ is the heredity ideal, it is also a projective $A(X,f)$-module, so this short exact sequence splits. Therefore $Q_1$ is a direct sum of copies of shifts of $P_Z$ by the Krull-Schmidt property. If $Q_1\neq 0$, then there must be a nonzero homomorphism from $P_Z$ to $Q_1$. But this is impossible by the same argument just used to prove $\Hom(Q_0,L(\pi_Z))=0$ above.


Now we consider $A(U,f)\cong Q_0$. Since we've just shown $A(U,f)$ is the quotient of $A(X,f)$ by the heredity ideal $J$, it is polynomial quasihereditary. Therefore by induction on the number of orbits we know that $\pi$ has even fibres over $U$. We've also shown that $\pi$ has even fibres over $Z$, completing the proof.
\end{proof}

%
%
%
%
%
%
%
%

\section{Application to quiver Hecke algebras}

Let $Q$ be a quiver of finite Dynkin type. This is a directed graph which becomes a finite type Dynkin diagram when the orientation is forgotten. Let $I$ be the set of vertices of $Q$. Fix $\nu=\sum_{i\in I} \nu_i i \in \N I$. Define
\[
 E_\nu = \prod_{i\to j} \Hom_{\C}(\C^{\nu_i},\C^{\nu_j})
\]
\[
 G_\nu = \prod_{i\in I} GL_{\nu_i}(\C).
\]
When we need to make the dependence on $Q$ explicit in the notation, we write $E_\nu(Q)$.
The stack $[E_\nu/G_\nu]$ is the moduli stack of representations of $Q$ of dimension vector $\nu$.

Let $Y_\nu$ be the variety of pairs $(x,\F^\bullet)$ where $x\in E_\nu$ and $\F^\bullet=(0=\F^0\subset \F^1\subset \cdots\subset \F^{|\nu|}=\oplus_{i\in I}\C^{\nu_i})$ is a full flag in $\oplus_{i\in I}\C^{\nu_i}$ satisfying the conditions $x\F^m\subset \F^m$ and $\F^m=\oplus_{i\in I}(\F^m\cap \C^{\nu_i})$ for all $m$. Let $f\map{Y_\nu}{E_\nu}$ be the $G_\nu$-equivariant morphism sending $(x,\F)$ to $x$.

The following definition of the quiver Hecke algebra is equivalent to the usual one via generators and relations \cite{maksimau}.

\begin{definition}
 The quiver Hecke algebra is the algebra
 \[
  R(\nu):= A(E_\nu,f). 
 \]
\end{definition}

Up to isomorphism, this algebra only depends on the underlying graph of $Q$, not its orientation. This fact follows from the presentation of $R(\nu)$ given in \cite{maksimau}, together with the discussion on change of parameters in \cite{kl2}.

We remark that $Y_\nu$ is smooth, disconnected and not equidimensional. The grading we have put on $R(\nu)$ differs in a trivial way from the usual one, which corresponds to replacing $\uk_{Y_\nu}$ by $IC(Y_\nu,k)$
in the definition of $R(\nu)$.

There is a canonical non-unital inclusion of algebras $R(\la)\otimes R(\mu)\hookrightarrow R(\la+\mu)$. Letting $e_{\la\mu}$ be the image of the unit under this embedding, we define a pair of adjoint induction and restriction functors by
\[
\Ind_{\la,\mu}(X)=R(\la+\mu)e_{\la\mu} \bigotimes_{R(\la)\otimes R(\mu)}X
\]
for $X\in R(\la)\otimes R(\mu)\mods$, and
\[
\Res_{\la,\mu}Y=e_{\la\mu}Y
\]
for $Y\in R(\la+\mu)\mods$.

If $M$ is a $R(\la)$-module and $N$ is a $R(\mu)$-module, we write $M\circ N$ for $\Ind_{\la\mu}(M\boxtimes N)$. The induction and restriction functor satisfy natural associativity properties and we define $M_1\circ\cdots \circ M_n$ and $\Res_{\la_1,\ldots,\la_n}$ accordingly.

Forgetting the orientation on $Q$ gives a Dynkin diagram of type ADE and hence a root system and Weyl group. Let $\{\a_i\}_{i\in I}$ be a corresponding set of simple roots, naturally indexed by $I$.
We identify $\a_i$ with $i$ and hence $\N I$ with the set of nonnegative linear combinations of simple roots. In this way we can talk meaningfully about roots as elements of $\N I$.

Write $\{s_i\}_{i\in I}$ for the generators of the Weyl group and $w_0$ for the longest element.
Let $N=\ell(w_0)$ and $w_0=s_{i_1}\cdots s_{i_N}$ be a reduced expression of $w_0$.
Define
\begin{equation}\label{def:bk}
 \b_k=s_{i_1}\cdots s_{i_{k-1}} \a_{i_k}.
\end{equation}
It is well known that $\b_1,\ldots,\b_N$ is an enumeration of the positive roots. 

Let $\ga_1,\ldots,\ga_N$ be a sequence of linear functionals $\ga_i\map{\N I}{\R}$ such that $\ga_i(\b_i)>0$, $\ga_i(\b_j)\leq 0$ if $i<j$ and $\ga_i(\b_j)\geq 0$ if $i>j$. We give two examples in (\ref{commutegamma}) and (\ref{def:ga}). The latter one is more important for our purposes.

 A \emph{Kostant partition} of $\nu\in\N I$ is a sequence of natural numbers $(n_1,\ldots,n_N)$ such that $\sum_i n_i\b_i =\nu$. Let $KP(\nu)$ be the set of Kostant partitions of $\nu$ and $\kpf(\nu)$ be the number of Kostant partitions of $\nu$.
 
 Define a partial order on the set of Kostant partitions of $\nu$ by
 $(n_1,\ldots,n_N)\preceq (m_1,\ldots,m_N)$ if
 \begin{equation}\label{sevenstar}
  \sum_{t=1}^k \ga_k(\b_t)  n_t \leq \sum_{t=1}^k  \ga_k(\b_t)  m_t
 \end{equation} for all $1\leq k\leq N$. This depends on the choice of $\ga_i$, but we suppress this dependence from the notation.
 
 
 Let $M$ be a $R(m\b_i)$-module, where $m\in \N$ and $\b_i$ is as in (\ref{def:bk}). $M$ is said to be \emph{semicuspidal} if for all $\la,\mu\in \N I$ such that $\Res_{\la,\mu}M\neq 0$, $\la\in \spann_{\N}\{\b_{i},\ldots,\b_N\}$ and $\mu\in \spann_{\N}\{b_1,\ldots,\b_i\}$.

 \begin{lemma}\label{lem:strong}
  Let $M_1,\ldots,M_N$ be semicuspidal representations of $R(m_1\b_1),\ldots, R(m_N\b_N)$ respectively. Suppose that $(n_1,\ldots,n_N)$ is a Kostant partition such that
  \[
   \Res_{n_1\b_1,\ldots,n_N\b_N} (M_1\circ\cdots\circ M_N)\neq 0.
  \]
Then $(m_1,\ldots ,m_N)\preceq (n_1,\ldots,n_N)$.
 \end{lemma}

 \begin{proof}
 We require the Mackey filtration for the composite $\Res\circ \Ind$ from \cite[Proposition 2.18]{khovanovlauda} whose most general statement is \cite[Proposition 2.7]{klr1}.
We consider the Mackey filtration for
  \[
   \Res_{n_1\b_1+\cdots+n_k\b_k,n_{k+1}\b_{k+1}+\cdots+n_N\b_N}(M_1\circ\cdots\circ M_N).
  \]
A nonzero subquotient is parametrised by $x_1,\ldots,x_N,y_1,\ldots y_N$ with $\Res_{x_i,y_i}M_i\neq 0$ and $n_1\b_1+\cdots+n_k\b_k=\sum_{i=1}^N x_i$.
If $i< k$, then as $M_i$ is semicuspidal, $y_i$ is a sum of roots $\b_j$ with $j\leq i$. So by the definition of $\ga_k$, $\ga_k(y_i)\leq 0$. 
If $i\geq k$ then again by semicuspidality, $x_i$ is a sum of roots $\b_j$ with $j\geq i$, so again by the definition of $\ga_k$, $\ga_k(x_i) \geq 0$ for such $i$.


 We have $m_i=x_i+y_i$ and $n_{k+1}\b_{k+1}+\cdots+n_N\b_N=y_1+\cdots+y_N$. Therefore
\[
\sum_{s=k+1}^N n_s\b_s-\sum_{s=k+1}^N m_s\b_s=(y_1+\cdots+y_k)-(x_{k+1}+\cdots+x_N).
\]
Apply $\ga_k$ to get $\ga_k(\sum_{i=1}^k m_i\b_i)\geq \ga_k(\sum_{i=1}^k n_i\b_i)$ for all $k$, as required.
\end{proof}


Lemma \ref{lem:strong} is precisely what is needed to run the arguments of \cite{klr1} and \cite{bkm} for the partial order $\prec$ rather than the bilexicographical order considered in those papers. We thus have the following corollary:
 
\begin{corollary}\label{refine}
 The results of \cite{klr1} and \cite{bkm}, as well as Theorem \ref{thm:klrhw} below, are valid with this more refined partial order replacing the bilexicographical order in those papers.
\end{corollary}

One example of an allowable choice of $\{\ga_k\}$ is
\begin{equation}\label{commutegamma}
 \ga_k=-s_{i_1}\cdots s_{i_k}\w^\vee_{i_k}.
\end{equation}
where the $\w^\vee_i$ are the fundamental coweights.

Two reduced expressions are said to be in the same commutation class if they can be reached from each other by only applying the commuting braid relations $s_is_j=s_js_i$. The following lemma is obvious.

\begin{lemma}
 If $\ii$ and $\jj$ are two reduced expressions in the same commutation class, then the corresponding partial orders defined by (\ref{sevenstar}) and (\ref{commutegamma}) on the set of Kostant partitions are the same.
\end{lemma}

As a Corollary, we recover the following c.f. \cite[Theorem 6.8]{ohsuh} and \cite[Theorem 5.10]{oh}.

\begin{corollary}\label{cor:comm}
 Two reduced expressions in the same commutation class induce the same polynomial highest weight structure on the category of representations of a quiver Hecke algebra.
\end{corollary}

\begin{remark}
 All of the results proven so far in this section, together with Theorem \ref{thm:klrhw} below, are purely algebraic, so hold for all finite type quiver Hecke algebras, not just the geometric ones considered in this paper.
\end{remark}

For geometric applications, we require a different choice of $\ga_i$. 
Let $\vv=\sum_{i\in I} v_i i$ and $\vw=\sum_{i\in I}w_i i$ be two elements of $\N I$. We define the Euler form by
\[
\langle \vv , \vw \rangle =\sum_{i\in I} v_i w_i - \sum_{i\to j}v_i w_j.
\]
Define
\begin{equation}\label{def:ga}
\ga_i(\vw)=\langle \vw,\b_i\rangle.
\end{equation}

\begin{theorem}[Gabriel's Theorem, \cite{gabriel}]\label{thm:gabriel}
 The function $M\mapsto \dim M$ induces a bijection between the set of isomorphism classes of indecomposable representations of $Q$ and the set $\Phi^+$ of positive roots.
\end{theorem}

For each $\a\in\Phi^+$, let $M(\a)$ be the indecomposable representation of $Q$ of dimension $\a$.

A reduced expression $s_{i_1}\cdots s_{i_l}$ is said to be \emph{adapted} to $Q$ if for each $1\leq k\leq l$, the vertex $i_k$ is a sink of the quiver $\s_{i_{k-1}}\cdots \s_{i_2}\s_{i_1}Q$. Here $\s_i$ is the operation on a quiver which reverses the direction of all arrows incident to the vertex $i$.
There exists a reduced expression of $w_0$ adapted to $Q$ \cite[Proposition 4.12(b)]{lusztigoneofthem}. Choose such a reduced expression and define the roots $\b_k$ by (\ref{def:bk}).

It is shown in \cite[p59]{ringel} that
\[
\dim\Hom(M(\b_k),M(\b_l))=\langle \b_k,\b_l\rangle \quad \text{if}\  k\leq l
\]
and
\[
\dim\Ext^1(M(\b_k),M(\b_l))=-\langle \b_k,\b_l\rangle \quad \text{if}\ k>l.
\]
In particular the Euler form satisfies $\langle\b_k,\b_l\rangle\geq 0$ if $k\leq l$ and $\langle \b_k,\b_l\rangle \leq 0$ if $k>l$, so defining $\{\ga_k\}$ as in (\ref{def:ga}) is admissible.

\begin{theorem}\cite[p59]{ringel}\label{triangularity}
Suppose the reduced expression $w_0=s_{i_1}\cdots s_{i_N}$ is adapted to $Q$. Then 
\[
 \dim\Hom(M(\b_k),M(\b_l))=\max( \ga_k(\b_l),0).
\]
\end{theorem}

Given a Kostant partition $\la=({n_1},\ldots,{n_N})$, define
\[
 M(\la)=\bigoplus_{i=1}^N M(\b_i)^{\oplus n_i}.
\]
There is a bijection between $KP(\nu)$ and the set of $G_\nu$-orbits on $E_\nu$ where the Kostant partition $\la$ is sent to the orbit of modules isomorphic to $M(\la)$.

The following is related to \cite[\S 4.3]{baumann}.

 \begin{theorem}
Suppose that the reduced expression $w_0=s_{i_1}\cdots s_{i_N}$ is adapted to $Q$ and consider the corresponding partial order $\prec$ on the set of Kostant partitions $KP(\nu)$, equivalently on the set of $G_\nu$-orbits on $E_\nu$. 
Then the partial order $\prec$ on $\KP(\nu)$ describes the closure order on the set of $G_\nu$-orbits on $E_\nu$.
\end{theorem}

 \begin{proof}
  By \cite[Proposition 3.2]{bongarz}, the orbit closure partial order is given by $M\preceq N$ if and only if
$\dim\Hom(M,X)\leq \dim\Hom(N,X)$
for all $X$. The result now follows from Theorem \ref{triangularity}.
  \end{proof}

%
%
%

\begin{theorem}\label{thm:checkheartsuit}
 The $G_\nu$-equivariant map $f\map{Y_\nu}{E_\nu}$ satisfies the conditions $(\spadesuit)_1$, $(\spadesuit)_2$ and $(\heartsuit)$, and $A(X,f)$ is even.
\end{theorem}

\begin{proof}
 Condition $(\spadesuit)_1$ is Gabriel's Theorem, Theorem \ref{thm:gabriel}.
 For $(\spadesuit)_2$, let  $\la=({n_1},\ldots,{n_N})$ be a Kostant partition. The stabiliser of a point in the orbit corresponding to $\la$ is the group of units $\End(M(\la))^\times$. By Theorem \ref{triangularity} and the property that $\ga_k(\b_l)\leq 0$ for $k>l$,
 there is a surjection
 \[
  \End(M(\la))^\times \twoheadrightarrow \prod_{i=1}^N \End(M(\a_i)^{\oplus n_i})^\times
 \]
 whose kernel is unipotent. 
Since $\End(M(\a))\cong k$ for all roots $\a$, this quotient group is isomorphic to $\prod_{i=1}^N GL_{n_i}$.
It is well-known that for $G$ a product of general linear groups, $H_G^*(\pt;k)$ is concentrated in even degrees. An extension by a unipotent group does not change the equivariant cohomology of a point, completing the proof of $(\spadesuit)_2$.

The evenness of $A(X,f)$ is due to Maksimau \cite[Theorem 1.1(1)]{maksimau}. Maksimau's theorem is actually stronger than what we need, which can be deduced by noticing that the argument in \cite[\S 5.3.6]{rouquier} works in all characteristics.


We now show $(\heartsuit)$. Le $\O$ be an orbit. 
By \cite[Theorem 2.2]{reineke}, 
there is a connected component $Y_\nu^\ii$ of $Y_\nu$, together with maps
\[
 Y_\nu^\ii \xrightarrow{h} Z \xrightarrow{g} \overline{\O}
\]
where $g$ is a resolution of singularities and $h$ is a fibration with fibre a product of flag varieties.
Therefore $h_*\uk$ is a direct sum of copies of shifts of the constant sheaf on $Z$. Since $g$ is a resolution, it is an isomorphism over $\O$, so $g_*\uk_Z$ restricts to the constant sheaf on $\O$. Therefore $(h\circ g)_*\uk$ has a summand supported on $\overline{\O}$ and restricting to the constant sheaf on $\O$. As $Y_\nu^\ii$ is a connected component of $Y_\nu$, the same is true for $f_*\uk_{Y_\nu}$.

By the main theorem of \cite{khovanovlauda}, the number of indecomposable projective $R(\nu)$-modules is equal to $\kpf(\nu)$. Thus the number of indecomposable summands of $f_*\uk_{Y_\nu}$ up to isomorphism and grading shift is also equal to $\kpf(\nu)$. The number of $G_\nu$-orbits on $E_\nu$ is also $\kpf(\nu)$. 
The previous paragraph shows that our desired map between indecomposable direct summands and orbits is surjective. Since the two sets have the same cardinality, this map is bijective, hence $(\heartsuit)$ holds.
\end{proof}

\begin{theorem}\label{thm:klrhw}
Let $R(\nu)$ be a quiver Hecke algebra of finite type. Then the category of $R(\nu)$-modules is polynomial highest weight for the partial order defined by \ref{sevenstar} (equivalently $R(\nu)$ is polynomial quasihereditary for this order).
\end{theorem}

\begin{proof}
First we show that the standard modules from \cite{bkm} are the same as in this paper. For now, we use $\D(\pi)$ to denote the standard modules defined in \cite[\S 3]{bkm}. To achieve this identification, we have to show that
\begin{enumerate}
 \item \label{cond:1} Any simple subquotient of $\D(\pi)$ is of the form $L(\sigma)$ with $\sigma\succeq \pi$.
 \item \label{cond:2} If $\sigma\succeq\pi$, then $\dim\Hom(\D(\pi),L(\sigma))=\delta_{\pi\sigma}$ and $\Ext^1(\D(\pi),L(\sigma))=0$.
\end{enumerate}

For (\ref{cond:1}), by 
\cite[Theorem 3.11]{bkm}, it suffices to consider composition factors $L(\sigma)$ of the proper standard modules, denoted $\overline{\D}(\pi)$ there. 
The desired triangularity result is now \cite[Theorem 3.1(5)]{klr1}, remembering Corollary \ref{refine}.

For (\ref{cond:2}), it follows in a standard manner from \cite[Theorem 3.12(2)]{bkm} and the upper-triangularity result we have just proved.

Therefore the standard modules in \cite{bkm} are the same as in this paper. We now proceed to check \ref{sc1}, \ref{sc2} and \ref{hwc}.

Condition \ref{sc1}, that a projective has a standard flag, is \cite[Corollary 3.14]{bkm}.

For the proof of \ref{sc2}, let $\la=(n_1,\ldots,n_N)$ be a Kostant partition. The standard module decomposes as an induced product
\[
 \D(\la)=\D(\b_1)^{(n_1)}\circ\cdots\circ \D(\b_N)^{(n_N)}.
\]
 Now consider the following computation
\begin{align*}
\End(\D(\la))&\cong \Hom(\D(\b_1)^{(n_1)}\otimes\cdots\otimes\D(\b_N)^{(n_N)},\Res_{n_1\b_1,\ldots, n_N\b_n}\D(\la)) \\ 
&\cong \End(\D(\b_1)^{(n_1)}\otimes\cdots\otimes\D(\b_N)^{(n_N)}) \\
&\cong \bigotimes_{i=1}^n \End(\D(\b_i)^{(n_i)}) \\
&\cong \bigotimes_{i=1}^n k[x_1,\ldots,x_{n_i}]^{S_{n_i}}.
\end{align*}
The first isomorphism is the induction-restriction adjunction, the second is the same proof as \cite[Lemma 3.3]{klr1} or better \cite[Lemma 8.6]{mcn3}. The final isomorphism follows from the content of \cite[\S 3.2]{bkm}, since the nil-Hecke algebra is a matrix algebra over symmetric functions.
Since the ring of symmetric functions is a polynomial algebra, we have proved \ref{sc2}.

Now we turn our attention to proving \ref{hwc} and first consider the case where the support of $\la$ is a single root $\a$, hence $\la$ is of the form $n\a$ for some $n\in \N$.
There is an algebra $S(n\a)$ which is the quotient of $R(n\a)$ such that the category of $S(n\a)$-modules is the same as the Serre subcategory generated by the simple module $L(\la)$. 
The module $\D(\la)$ is an indecomposable projective module over $S(n\a)$. 
By Morita theory, since $S(n\a)$ only has one simple module, $S(n\a)\cong \operatorname{Mat}_{\dim L(\la)}(\End\D(\la))$. As $\End(\D(\la))$ is commutative, $\D(\la)$ is free and of finite rank over its endomorphism algebra.

Now for general $\la=(n_1,\ldots,n_N)$, we have
\[
 \D(\la)\cong \bigoplus_w \psi_w \D(\b_1)^{(n_1)}\otimes \cdots\otimes\D(\b_n)^{(n_N)}
\] as vector spaces, where $\psi_w$ is a finite collection of elements in $R(\nu)$.

For $\phi=\phi_1\otimes\cdots\otimes \phi_N\in \End(\D(\la))\cong \otimes_i \End(\D(\b_i)^{(n_i)})$, we have
\[
\phi( \psi_w (v_1\otimes\cdots\otimes v_n)) = \psi_w \phi_1(v_1)\otimes\cdots\otimes \phi_n(v_n)
\]
Since each $\D(\b_i)^{(n_i)}$ has just been shown to be finite free over $\End(\D(\b_i)^{(n_i)})$, each summand $\psi_w \D(\b_1)^{(n_1)}\otimes \cdots\otimes\D(\b_n)^{(n_N)}$ is finite free over $\otimes_i \End(\D(\b_i)^{(n_i)})\cong \End(\D(\la))$, and hence the same is true for $\D(\la)$, as required.
\end{proof}

We are now able to apply Theorem \ref{reverse} to deduce the following corollary. It has a direct proof in type A due to Maksimau \cite[Corollary 3.36]{maksimau} and is new in types D and E.

\begin{corollary}\label{cor:feven}
 The morphism $f\map{Y_\nu}{E_\nu}$ is even.
\end{corollary}

An example of an immediate consequence of this Corollary is that \cite[Theorem 3.7]{james} is now known to hold in all finite ADE types.

%
%

\section{Reflection functors}

We now generalise Kato's theory of reflection functors for quiver Hecke algebras developed in \cite{katopbw} from characteristic zero to all characteristics. These are functors which categorify Lusztig's action of the braid group on $U_q(\g)$ by algebra automorphisms $T_i$. 

Let $i\in I$. Choose an orientation of $Q$ such that $i$ is a sink. Let $Q^\sigma$ be the quiver obtained from $Q$ by reversing the direction of all arrows incident to $i$. Given a representation $V$ of either $Q$ or $Q^\sigma$, write $V_j$ for the vector space placed at the vertex $j$.

Let $U_i(\nu)$ be the subset of $E_\nu(Q)$ consisting of modules for which the map
\[
 \bigoplus_{j\to i} V_j \xrightarrow{\phi} V_i
\]
is surjective.

Let $_iU(\nu)$ be the subset of $E_\nu(Q^\sigma)$ consisting of modules for which the map
\[
 V_i \xrightarrow{\psi} \bigoplus_{i\to j} V_j
\]
is injective.

\begin{theorem}\cite[Theorem 1.1]{bgp}
 The BGP reflection functor $T_i\map{Q\mods}{Q^\sigma\mods}$ induces an equivalence between the full subcategory of $Q\mods$ where $\phi$ is surjective, and the full subcategory of $^\sigma Q\mods$ where $\psi$ is injective.
\end{theorem}

\begin{corollary}\label{stackiso}
 The quotient stacks $[U_i(\nu)/G_\nu]$ and $[_iU(s_i\nu)/G_{s_i\nu}]$ are isomorphic.
\end{corollary}

\begin{proof}
 These are the moduli stacks of objects in the categories shown to be equivalent in the above theorem.
\end{proof}

%
%

\begin{lemma}
The map $A(E_\nu(Q),f)\to A(U_i(\nu),f)$ is surjective. 
\end{lemma}

\begin{proof}
This is an immediate consequence of Corollaries \ref{cor:feven} and \ref{cor:fsurj}.
\end{proof}

The positive part of the quantum group $U^+$ is the unital associative algebra generated by elements $\th_i$ subject to the quantum Serre relations
\begin{align*}
\th_i\th_j&=\th_j\th_i &\mbox{ if } i\cdot j&=0 \\
(q+q\inv)\th_i\th_j\th_i &= \th_i\th_j^2+\th_j\th_i^2 &\mbox{ if } i\cdot j &= -1
\end{align*}
The algebra $U^+$ has an integral form over $\Z[q,q\inv]$ which we call $\f$. It is the $\Z[q,q\inv]$-subalgebra generated by the divided powers $\th_i^n/[n]!$, where $[n]!=\prod_{i=1}^n\frac{q^i-q^{-i}}{q-q\inv}$. This algebra is graded by $\N I$. The main theorem of \cite{khovanovlauda} is that the quiver Hecke algebras categorify $\f$, which means there is an isomorphism
\[
\bigoplus_\nu K_0(R(\nu)\prmod)\cong \f.
\]
Here $R(\nu)\prmod$ is the category of finitely generated graded projective $R(\nu)$-modules. Further properties of this isomorphism including compatibilities with induction and restriction functors are discussed and proven in \cite{khovanovlauda}.

Let $i\in I$. Let $r_i\map{\f}{\f}$ be the linear map defined inductively by
\begin{align*}
r_i(\th_j)&=\delta_{ij} \\
r_i(xy)&=q^{\deg(y)\cdot i} r_i(x)y+x r_i(y).
\end{align*}

For our purposes we need to know that $r_i$ categorifies the functor $\Res_{i,\nu-i}$ in the sense that
\begin{equation}\label{ricat}
[\Res_{i,\nu-i}M]=\theta_i\otimes r_i([M]).
\end{equation}
This follows from the results in \cite{khovanovlauda}.

We also need the decomposition $\f\cong \ker(r_i)\oplus \f \th_i$. \cite[\S 38.1]{bookoflusztig}.

\begin{proposition}
The inclusion 
 $A(U_i(\nu),f)\mods\subset R(\nu)\mods$ induces the inclusion $\ker(r_i)\subset \f$ upon passing to Grothendieck groups.
\end{proposition}

\begin{proof}
Let $j\map{U_i(\nu)}{E_\nu}$ be the inclusion. It is an open immersion. Let $z$ be the inclusion of the closed complement. If $\P$ is an indecomposable parity sheaf, then $j^*\P$ is either indecomposable or zero, by \cite[Proposition 2.11]{parity}. 
Therefore the collection of indecomposable parity sheaves on $U_i(\nu)$ is exactly the restrictions of those indecomposable parity sheaves on $E_\nu$ which have a nonzero restriction. 
 
Let $\P$ be such an indecomposable parity sheaf and consider the short exact sequence from Lemma \ref{parityses}:
 \begin{equation}\label{rises}
  0 \to \homb(\L,z_!z^!\P)\to \homb(\L,\P)\to \homb(j^*\L,j^*\P)\to 0.
 \end{equation}
We will show that this short exact sequence categorifies the decomposition $\f\cong \ker(r_i)\oplus \f\th_i$. 

Since $\P$ is parity, a short computation using base change shows that $z_!z^!\P$ is $!$-parity. Therefore by Lemma \ref{standardflag}, $\homb(\L,z_!z^!\P)$ has a standard flag. Each standard module appearing in this flag is supported on $E_\nu\setminus U_i(\nu)$, thus corresponds to a Kostant partition where $n_N\neq 0$. Therefore the standard module is of the form $\D'\circ R(i)$ (up to multiplicity). Therefore $[\homb(\L,z_!z^!\P)]$ lies in $\f\th_i$.

We now show that the class of the $A(U_i(\nu),f)$-module $\homb(j^*\L,j^*\P)$ lies in $\ker(r_i)$.
Using (\ref{ricat}), we have to show that if $L$ is a simple module for $A(U_i(\nu),f)$, then $\Res_{i,\nu-i}L=0$.

 
 Suppose for want of a contradiction that $L$ is simple for $A(U_i(\nu),f)$ and $\Res_{i,\nu-i}L\neq 0$. Then there exists a projective $R(\nu-i)$-module $Q$ together with a nonzero map from $R(i)\otimes Q$ to $\Res_{i,\nu-i}L$ and hence by adjunction a nonzero map from $R(i)\circ Q$ to $L$.
 
 The projective module $R(i)\circ Q$ is of the form $\Hom^\bullet(\L,\P)$ for some $\P$. As $\P$ is induced, its support is contained in the set of modules which have $S_i$ as a quotient, where $S_i$ is the simple representation of $Q$ at the vertex $i$. Since $i$ is a sink, $S_i$ is a projective representation of $Q$. Therefore $\supp(\P)$ is contained in the set of modules
 of the form $S_i\oplus X$. Letting $Z$ be the complement of $U_i(\nu)$ in $E_\nu$, we see that $\supp(\P)\subset Z$. Now as in the proof of Theorem \ref{reverse}, there cannot be any nonzero homomorphism from $R(i)\circ Q$ to $L$ (as $j^*z_*=0$). This is a contradiction and hence we've shown that $[\homb(j^*\L,j^*\P)]$ lies in $\ker(r_i)$.
 
 
 Turning our attention back to the short exact sequence (\ref{rises}), we have the identity in $\f$
 \[
  [\homb(\L,\P)]= [\homb(j^*\L,j^*\P)]+[ \homb(\L,z_!z^!\P)].
 \]
We have shown that $[\homb(j^*\L,j^*\P)]\in \ker(r_i)$ and $[ \homb(\L,z_!z^!\P)]\in \f \th_i$. The short exact sequence (\ref{rises}) is thus a categorification of the direct sum decomposition $\f\cong \ker(r_i)\oplus \f\th_i$.

Let $\tau$ be the projection from $\f$ to $\ker(r_i)$. We've shown that $\tau([\homb(\L,\P)])= [\homb(j^*\L,j^*\P)]$. Since the elements of the form $\tau([\homb(\L,\P)])$ span $\ker(r_i)$ and those of the form $[\homb(j^*\L,j^*\P)]$ span $K_0(A(U_i(\nu),f))$, we've shown the desired isomorphism $K_0(A(U_i(\nu),f))\cong \ker(r_i)$ in a way which is compatible with the inclusion into $\f$.
\end{proof}

\begin{theorem}\label{thm:morita}
The algebras $ A(U_i(\nu),f)$ and $ A(_iU(s_i\nu),f)$ are Morita equivalent.
\end{theorem}

\begin{proof}
Let $\P_i$ be the direct sum of all indecomposable parity sheaves (up to shift) on $[U_i(\nu)/G_\nu]$. Let $_i\P$ be the direct sum of all indecomposable parity sheaves (up to shift) on $[_iU(s_i\nu)/G_{s_i\nu}]$. Normalise $\P_i$ and $_i\P$ so that they are Verdier self-dual, this is possible because Verdier duality preserves the set of parity sheaves. By Corollary \ref{stackiso}, there is an isomorphism of algebras
\[
 \homb(\P_i,\P_i)\cong \homb(_i\P,{_i}\P).
\]
The indecomposable summands of $f_*\uk_{Y_i}\in D^b_{G_\nu}(U_i(\nu))$ up to shift comprise the set of all indecomposable $G_\nu$-equivariant parity sheaves on $U_i(\nu)$ up to shift. Therefore the algebras $A(U_i(\nu),f)$ and $\homb(\P_i,\P_i)$ are Morita equivalent. (An explicit bimodule giving this Morita equivalence is $\homb(f_*\uk_{Y_i},\P_i)$, where $Y_i=f\inv(U_i(\nu))\subset Y_\nu$.)

Similarly $ A(_iU(s_i\nu),f)$ and $\homb(_i\P,{_i}\P)$ are Morita equivalent. This proves the theorem.
\end{proof}

\begin{definition}
 Let $\T_i$ denote the equivalence from $ A(U_i(\nu),f)\mods$ to  $ A(_iU(s_i\nu),f)\mods$ constructed in Theorem \ref{thm:morita}.
\end{definition}

We now entertain a discussion as to whether or not this functor $\T_i$ is monoidal, in the sense that there is a natural isomorphism $\T_i(M\circ N)\cong \T_i(M)\circ \T_i(N)$ for all modules $M,N\in A(U_i(\nu),f)\mods$.
If the algebra $A(X,f)$ is formal as a dg-algebra then since $A(X,f)$ is of finite global dimension, every finitely generated module is of the form $M\cong \Hom^\bullet(\L,\F)$ for some $\F\in D^b_G(X;k)$. Then monoidality of $\T_i$ follows from the corresponding statement at the sheaf level, which is easy to check. This check appears in \cite{kato:monoidal,geom:monoidal}, from which monoidality is deduced in characteristic zero via the formality of $A(X,f)$. Unfortunately, formality of $A(X,f)$ is currently unknown when $k$ has positive characteristic. 


%
%
%
%
%
For $i\in I$ there is an automorphism $T_i$ of the entire quantum group (denoted $T'_{i,+}$ in \cite{bookoflusztig}). It is customarily used to construct PBW bases. When restricted to $\f\subset U_q(\g)$, this induces an isomorphism of subalgebras $T_i\map{\ker(r_i)}{\ker(_ir)}$ \cite[Prop 38.1.6]{bookoflusztig}, where $_ir$ is defined like $r_i$ but with multiplication in the other order, so that it is the decategorification of $\Res_{\nu-i,i}$. Our final result is that our equivalence $\T_i$ decategorifies to $T_i$. 

 \begin{theorem}
  The equivalence of categories $\T_i \map{A(U_i(\nu),f)\mods}{A(_iU(s_i\nu),f)\mods}$ induces the isomorphism $T_i\map{\ker(r_i)}{\ker(_ir)}$ at the level of Grothendieck groups.
 \end{theorem}

 \begin{proof}
 The classes of the standard modules comprise a basis for the Grothendieck groups. Since the standard modules are intrinsic to the polynomial highest weight category with its partial order, $\T_i$ sends a standard module to a standard module.
 
 By \cite[\S 3]{bkm}, the classes of the standard modules are elements of the PBW basis. The isomorphism $T_i\map{\ker(r_i)}{\ker(_ir)}$ maps one PBW basis to the other. This completes the proof up to powers of $q$ and a permutation.
 By ensuring that our Morita equivalences always use the appropriate self-dual sheaves, we ensure that there are no powers of $q$ to deal with. To confirm there is no permutation, we consider the geometric description of standard modules to see that each standard module gets sent to the appropriate standard module under $\T_i$.
 \end{proof}

%
%
%
%
%


\begin{thebibliography}{JMS}

 \bibitem[Ba]{baumann}
 Pierre Baumann, The canonical basis and the quantum Frobenius morphism. \arxiv{1201.0303}

 
\bibitem[BGP]{bgp}
Bernstein, Gel'fand and Ponomarev, \newblock Coxeter functors, and {G}abriel's theorem.
\newblock {\em Uspehi Mat. Nauk}, 28(2(170)):19--33, 1973.

\bibitem[BL]{bernsteinlunts}
Joseph Bernstein and Valery Lunts.
\newblock {\em Equivariant sheaves and functors}, volume 1578 of {\em Lecture
  Notes in Mathematics}.
\newblock Springer-Verlag, Berlin, 1994.


\bibitem[Bo]{bongarz}
Klaus Bongartz, On degenerations and extensions of finite dimensional modules.
{\it Adv. Math.} {\bf 121} (1996), no 2, 245--287.

\bibitem[BKM]{bkm}
Jon Brundan, Alexander Kleshchev and Peter J. McNamara, Homological Properties of Finite Type Khovanov-Lauda-Rouquier Algebras. 
{\it Duke Math. J.}  {\bf 163} (2014), no. 7, 1353--1404.
\arXiv{1210.6900}
 
\bibitem[CG]{chrissginzburg}
Neil Chriss and Victor Ginzburg.
\newblock {\em Representation theory and complex geometry}.
\newblock Birkh\"auser Boston, Inc., Boston, MA, 1997.
 
\bibitem[CPS]{cps}
E.~Cline, B.~Parshall, and L.~Scott.
\newblock Finite-dimensional algebras and highest weight categories.
\newblock {\em J. Reine Angew. Math.}, 391:85--99, 1988.

\bibitem[F]{fujita}
Ryo Fujita.
\newblock Tilting modules of affine quasi-hereditary algebras.
\newblock {\em Adv. Math.}, 324:241--266, 2018.
\arxiv{1610.02621}

\bibitem[G]{gabriel}
Peter Gabriel.
\newblock Unzerlegbare {D}arstellungen. {I}.
\newblock {\em Manuscripta Math.}, 6:71--103; correction, ibid. 6 (1972), 309,
  1972.
  
\bibitem[JMW]{parity}
Daniel Juteau, Carl Mautner, and Geordie Williamson.
\newblock Parity sheaves.
\newblock {\em J. Amer. Math. Soc.}, 27(4):1169--1212, 2014. \arxiv{0906.2994}


\bibitem[K1]{katoext}
Syu Kato.
\newblock An algebraic study of extension algebras.
\newblock {\em Amer. J. Math.}, 139(3):567--615, 2017.
\arXiv{1207.4640}

\bibitem[K2]{katopbw}
Syu Kato.
\newblock Poincar\'e-{B}irkhoff-{W}itt bases and {K}hovanov-{L}auda-{R}ouquier
  algebras.
\newblock {\em Duke Math. J.}, 163(3):619--663, 2014. \arXiv{1203.5254}

\bibitem[K3]{kato:monoidal}
Syu Kato.
\newblock On the monoidality of Saito reflection functors. 
\newblock {\em Int. Math. Res. Not.}, rny233, 2018.
\arxiv{1711.09085}

\bibitem[KL1]{khovanovlauda}
Mikhail Khovanov and Aaron~D. Lauda.
\newblock A diagrammatic approach to categorification of quantum groups. {I}.
\newblock {\em Represent. Theory}, 13:309--347, 2009.
\arxiv{0803.4121}

\bibitem[KL2]{kl2}
Mikhail Khovanov and Aaron~D. Lauda.
\newblock A diagrammatic approach to categorification of quantum groups {II}.
\newblock {\em Trans. Amer. Math. Soc.}, 363(5):2685--2700, 2011.
\arxiv{0804.2080}

\bibitem[Kl]{kleshchev}
Alexander~S. Kleshchev.
\newblock Affine highest weight categories and affine quasihereditary algebras.
\newblock {\em Proc. Lond. Math. Soc. (3)}, 110(4):841--882, 2015. \arxiv{1405.3328}

\bibitem[L1]{lusztig}
George Lusztig.
\newblock Cuspidal local systems and graded {H}ecke algebras. {I}.
\newblock {\em Inst. Hautes \'Etudes Sci. Publ. Math.}, (67):145--202, 1988.

 \bibitem[L2]{lusztigoneofthem}
G.~Lusztig.
\newblock Canonical bases arising from quantized enveloping algebras.
\newblock {\em J. Amer. Math. Soc.}, 3(2):447--498, 1990.

 
 \bibitem[L3]{bookoflusztig}
George Lusztig.
\newblock {\em Introduction to quantum groups}, volume 110 of {\em Progress in
  Mathematics}.
\newblock Birkh\"auser Boston Inc., Boston, MA, 1993.

 \bibitem[M1]{maksimau}
Ruslan Maksimau.
\newblock Canonical basis, {KLR} algebras and parity sheaves.
\newblock {\em J. Algebra}, 422:563--610, 2015.
 \arXiv{1301.6261}
 
 \bibitem[M2]{rm19}
 Ruslan Maksimau.
 \newblock Flag versions of quiver Grassmanniansfor Dynkin quivers have no odd cohomology over $\mathbb{Z}$.
 \arxiv{1909.04907}
 
 \bibitem[Mc1]{klr1}
 Peter~J. McNamara.
\newblock Finite dimensional representations of {K}hovanov-{L}auda-{R}ouquier
  algebras {I}: {F}inite type.
\newblock {\em J. Reine Angew. Math.}, 707:103--124, 2015.
 \arxiv{1207.5860}
 
 \bibitem[Mc2]{mcn3}
 Peter J. McNamara, Representations of Khovanov-Lauda-Rouquier algebras III: Symmetric Affine Type, to appear, Math Z.  \arxiv{1407.7304}
 
 \bibitem[Mc3]{geom:monoidal}
 Peter J. McNamara, Monoidality of Kato's reflection functors. \arxiv{1712.00173}
 
 
\bibitem[O]{oh}
Se-jin Oh.
\newblock Auslander-{R}eiten quiver and representation theories related to
  {KLR}-type {S}chur-{W}eyl duality.
\newblock {\em Math. Z.}, 291(1-2):499--554, 2019.
\arxiv{1509.04949} 
 
 \bibitem[OS]{ohsuh}
Se-jin Oh and Uhi~Rinn Suh.
\newblock Combinatorial {A}uslander-{R}eiten quivers and reduced expressions.
\newblock {\em J. Korean Math. Soc.}, 56(2):353--385, 2019.
\arxiv{1509.04820}



 
\bibitem[Re]{reineke}
Markus Reineke.
\newblock Quivers, desingularizations and canonical bases.
\newblock In {\em Studies in memory of {I}ssai {S}chur ({C}hevaleret/{R}ehovot,
  2000)}, volume 210 of {\em Progr. Math.}, pages 325--344. Birkh\"auser
  Boston, Boston, MA, 2003. \arxiv{math/0104284}

\bibitem[Ri]{ringel}
Claus~Michael Ringel.
\newblock P{BW}-bases of quantum groups.
\newblock {\em J. Reine Angew. Math.}, 470:51--88, 1996.

\bibitem[Ro]{rouquier}
Rapha\"el Rouquier.
\newblock {Quiver Hecke algebras and 2-Lie algebras.}
\newblock {\em Algebra Colloq.}, 19(2):359--410, 2012.
\arxiv{1112.3619}

\bibitem[We]{webster}
Ben Webster.
\newblock Weighted {K}hovanov-{L}auda-{R}ouquier algebras.
\newblock {\em Doc. Math.}, 24:209--250, 2019.
\arxiv{1209.2463}

\bibitem[Wi]{james}
Geordie Williamson.
\newblock On an analogue of the {J}ames conjecture.
\newblock {\em Represent. Theory}, 18:15--27, 2014.
 \arxiv{1212.0794}

 

\end{thebibliography}
\end{document}